\documentclass{amsart}
\usepackage{amssymb}
\usepackage{amsfonts}

\newtheorem{theorem}{Theorem}
\theoremstyle{plain}

\newtheorem{corollary}{Corollary}

\newtheorem{definition}{Definition}

\newtheorem{lemma}{Lemma}

\newtheorem{remark}{Remark}

\numberwithin{equation}{section}
\input{tcilatex}

\begin{document}
\title[Hermite-Hadamard-type inequalities]{Hermite-Hadamard-type
inequalities for $\left( g,\varphi _{h}\right) -$ convex dominated functions}
\author{M. Emin \"{O}zdemir}
\address{Atat\"{u}rk University, K.K. Education Faculty, Department of
Mathematics, 25240 Campus, Erzurum, Turkey}
\email{emos@atauni.edu.tr}
\author{Mustafa G\"{u}rb\"{u}z$^{\blacksquare }$}
\address{A\u{g}r\i\ \.{I}brahim \c{C}e\c{c}en University, Faculty of
Education, Department of Mathematics, 04100, A\u{g}r\i , Turkey}
\email{mgurbuz@agri.edu.tr}
\thanks{Corresponding Author$^{\blacksquare }$}
\author{Havva Kavurmac\i }
\address{A\u{g}r\i\ \.{I}brahim \c{C}e\c{c}en University, Faculty of Science
and Letters, Department of Mathematics, 04100, A\u{g}r\i , Turkey}
\email{hkavurmaci@agri.edu.tr}
\date{July 17, 2012}
\subjclass[2000]{ Primary 26D15, Secondary 26D10, 05C38}
\keywords{Convex dominated functions, Hermite-Hadamard Inequality, $\varphi
_{h}-$convex functions, $\left( g,s\right) -$convex dominated functions}

\begin{abstract}
In this paper, we introduce the notion of $\left( g,\varphi _{h}\right) -$%
convex dominated function and present some properties of them. Finally, we
present a version of Hermite-Hadamard-type inequalities for $\left(
g,\varphi _{h}\right) -$convex dominated functions. Our results generalize
the Hermite-Hadamard-type inequalities in [2], [4] and [6].
\end{abstract}

\maketitle

\section{Introduction}

The inequality%
\begin{equation}
f\left( \frac{a+b}{2}\right) \leq \frac{1}{b-a}\int_{a}^{b}f\left( x\right)
dx\leq \frac{f\left( a\right) +f\left( b\right) }{2}  \label{h}
\end{equation}%
which holds for all convex functions $f:[a,b]\rightarrow 
%TCIMACRO{\U{211d} }%
%BeginExpansion
\mathbb{R}
%EndExpansion
$, is known in the literature as Hermite-Hadamard's inequality.

In \cite{DI}, Dragomir and Ionescu introduced the following class of
functions.

\begin{definition}
Let $g:I\rightarrow 
%TCIMACRO{\U{211d} }%
%BeginExpansion
\mathbb{R}
%EndExpansion
$ be a convex function on the interval $I.$ The function $f:I\rightarrow 
%TCIMACRO{\U{211d} }%
%BeginExpansion
\mathbb{R}
%EndExpansion
$ is called $g-$convex dominated on $I$ if the following condition is
satisfied:%
\begin{eqnarray*}
&&\left\vert \lambda f\left( x\right) +\left( 1-\lambda \right) f\left(
y\right) -f\left( \lambda x+\left( 1-\lambda \right) y\right) \right\vert \\
&& \\
&\leq &\lambda g\left( x\right) +\left( 1-\lambda \right) g\left( y\right)
-g\left( \lambda x+\left( 1-\lambda \right) y\right)
\end{eqnarray*}%
for all $x,y\in I$ and $\lambda \in \left[ 0,1\right] .$
\end{definition}

In \cite{DPP}, Dragomir \textit{et al.} proved the following theorem for $g-$%
convex dominated functions related to (\ref{h}).

\label{s} Let $g:I\rightarrow 
%TCIMACRO{\U{211d} }%
%BeginExpansion
\mathbb{R}
%EndExpansion
$ be a convex fuction and $f:I\rightarrow 
%TCIMACRO{\U{211d} }%
%BeginExpansion
\mathbb{R}
%EndExpansion
$ be a $g-$convex dominated mapping. Then, for all $a,b\in I$ with $a<b,$ 
\begin{equation*}
\left\vert f\left( \frac{a+b}{2}\right) -\frac{1}{b-a}\int_{a}^{b}f\left(
x\right) dx\right\vert \leq \frac{1}{b-a}\int_{a}^{b}g\left( x\right)
dx-g\left( \frac{a+b}{2}\right)
\end{equation*}%
and%
\begin{equation*}
\left\vert \frac{f\left( a\right) +f\left( b\right) }{2}-\frac{1}{b-a}%
\int_{a}^{b}f\left( x\right) dx\right\vert \leq \frac{g\left( a\right)
+g\left( b\right) }{2}-\frac{1}{b-a}\int_{a}^{b}g\left( x\right) dx.
\end{equation*}%
In \cite{DI} and \cite{DPP}, the authors connect together some disparate
threads through a Hermite-Hadamard motif. The first of these threads is the
unifying concept of a $g-$convex-dominated function. In \cite{HHW}, Hwang 
\textit{et al.} established some inequalities of Fej\'{e}r type for $g-$%
convex-dominated functions. Finally, in \cite{KOS}\textit{, }\cite{OKT} and 
\cite{OTK} authors introduced several new different kinds of convex
-dominated functions and then gave Hermite-Hadamard-type inequalities for
this classes of functions.

In \cite{S}, S. Varo\v{s}anec introduced the following class of functions.

$I$ and $J$ are intervals in $%
%TCIMACRO{\U{211d} }%
%BeginExpansion
\mathbb{R}
%EndExpansion
,$ $\left( 0,1\right) \subseteq J$ and functions $h$ and $f$ are real
non-negative functions defined on $J$ and $I$, respectively.

\begin{definition}
Let $h:J\rightarrow 
%TCIMACRO{\U{211d} }%
%BeginExpansion
\mathbb{R}
%EndExpansion
$ be a non-negative function, $h\not\equiv 0$. We say that $f:I\rightarrow 
%TCIMACRO{\U{211d} }%
%BeginExpansion
\mathbb{R}
%EndExpansion
$ is an $h-$convex function, or that $f$ belongs to the class $SX\left(
h,I\right) ,$ if $f$ is non-negative and for all $x,y\in I,$ $\alpha \in
(0,1],$ we have%
\begin{equation}
f(\alpha x+\left( 1-\alpha \right) y)\leq h\left( \alpha \right)
f(x)+h\left( 1-\alpha \right) f(y).  \label{z3}
\end{equation}
\end{definition}

If the inequality (\ref{z3}) is reversed, then $f$ is said to be $h-$%
concave, i.e. $f\in SV\left( h,I\right) .$

Youness have defined the $\varphi -$convex functions in \cite{Youness}. A
function $\varphi :\left[ a,b\right] \rightarrow \left[ c,d\right] $ where $%
\left[ a,b\right] \subset 
%TCIMACRO{\U{211d} }%
%BeginExpansion
\mathbb{R}
%EndExpansion
$:

\begin{definition}
A function $f:\left[ a,b\right] \rightarrow 
%TCIMACRO{\U{211d} }%
%BeginExpansion
\mathbb{R}
%EndExpansion
$ is said to be $\varphi -$convex on $\left[ a,b\right] $ if for every two
points $x\in \left[ a,b\right] ,$ $y\in \left[ a,b\right] $ and $t\in \left[
0,1\right] $ the following inequality holds:%
\begin{equation*}
f\left( t\varphi \left( x\right) +\left( 1-t\right) \varphi \left( y\right)
\right) \leq tf\left( \varphi \left( x\right) \right) +\left( 1-t\right)
f\left( \varphi \left( y\right) \right) .
\end{equation*}
\end{definition}

In \cite{S1}, Sar\i kaya defined a new kind of $\varphi -$convexity using $%
h- $convexity as following:

\begin{definition}
Let I be an interval in $%
%TCIMACRO{\U{211d} }%
%BeginExpansion
\mathbb{R}
%EndExpansion
$ and $h:\left( 0,1\right) \rightarrow \left( 0,\infty \right) $ be a given
function. We say that a function $f:I\rightarrow \lbrack 0,\infty )$ is $%
\varphi _{h}-$convex if%
\begin{equation}
f\left( t\varphi \left( x\right) +\left( 1-t\right) \varphi \left( y\right)
\right) \leq h\left( t\right) f\left( \varphi \left( x\right) \right)
+h\left( 1-t\right) f\left( \varphi \left( y\right) \right)  \label{a}
\end{equation}%
for all $x,y\in I$ and $t\in \left( 0,1\right) .$
\end{definition}

If inequality (\ref{a}) is reversed, then $f$ is said to be $\varphi _{h}-$%
concave. In particular, if $f$ satisfies (\ref{a}) with $h\left( t\right)
=t, $ $h\left( t\right) =t^{s}$ $\left( s\in \left( 0,1\right) \right) ,$ $%
h\left( t\right) =\frac{1}{t},$ and $h\left( t\right) =1,$ then $f$ is said
to be $\varphi -$convex, $\varphi _{s}-$convex, $\varphi -$Godunova-Levin
function and $\varphi -P-$function, respectively.

In the following sections our main results are given: We introduce the
notion of $\left( g,\varphi _{h}\right) -$convex dominated function and
present some properties of them. Finally, we present a version of
Hermite-Hadamard-type inequalities for $\left( g,\varphi _{h}\right) -$%
convex dominated functions. Our results generalize the Hermite-Hadamard-type
inequalities in \cite{DPP}, \cite{KOS} and \cite{OTK}.

\section{$\left( g,\protect\varphi _{h}\right) -$convex dominated functions}

\begin{definition}
\label{1} Let $h:\left( 0,1\right) \rightarrow \left( 0,\infty \right) $ be
a given function, $g:I\rightarrow \lbrack 0,\infty )$ be a given $\varphi
_{h}-$convex function. The real function $f:I\rightarrow \lbrack 0,\infty )$
is called $\left( g,\varphi _{h}\right) -$convex dominated on $I$ if the
following condition is satisfied%
\begin{eqnarray}
&&\left\vert h\left( t\right) f\left( \varphi \left( x\right) \right)
+h\left( 1-t\right) f\left( \varphi \left( y\right) \right) -f\left(
t\varphi \left( x\right) +\left( 1-t\right) \varphi \left( y\right) \right)
\right\vert  \label{b} \\
&&  \notag \\
&\leq &h\left( t\right) g\left( \varphi \left( x\right) \right) +h\left(
1-t\right) g\left( \varphi \left( y\right) \right) -g\left( t\varphi \left(
x\right) +\left( 1-t\right) \varphi \left( y\right) \right)  \notag
\end{eqnarray}%
for all $x,y\in I$ and $t\in (0,1)$.
\end{definition}

In particular, if $f$ satisfies (\ref{b}) with $h\left( t\right) =t,$ $%
h\left( t\right) =t^{s}$ $\left( s\in \left( 0,1\right) \right) ,$ $h\left(
t\right) =\frac{1}{t}$ and $h\left( t\right) =1,$ then $f$ is said to be $%
\left( g,\varphi \right) -$convex-dominated, $\left( g,\varphi _{s}\right) -$%
convex-dominated, $\left( g,\varphi _{Q(I)}\right) -$convex-dominted and $%
\left( g,\varphi _{P(I)}\right) -$convex-dominated functions, respectively.

The next simple characterisation of $\left( g,\varphi _{h}\right) -$convex
dominated functions holds.

\begin{lemma}
\label{2} Let $h:\left( 0,1\right) \rightarrow \left( 0,\infty \right) $ be
a given function, $g:I\rightarrow \lbrack 0,\infty )$ be a given $\varphi
_{h}-$convex function and $f:I\rightarrow \lbrack 0,\infty )$ be a real
function. The following statements are equivalent:
\end{lemma}

\begin{enumerate}
\item $f$ is $\left( g,\varphi _{h}\right) -$convex dominated on $I.$

\item The mappings $g-f$ and $g+f$ are $\varphi _{h}-$ convex on $I.$

\item There exist two $\varphi _{h}-$convex mappings $l,k$ defined on $I$
such that%
\begin{equation*}
\begin{array}{ccc}
f=\frac{1}{2}\left( l-k\right) & \text{and} & g=\frac{1}{2}\left( l+k\right)%
\end{array}%
.
\end{equation*}
\end{enumerate}

\begin{proof}
1$\Longleftrightarrow $2 The condition (\ref{b}) is equivalent to%
\begin{eqnarray*}
&&g\left( t\varphi \left( x\right) +\left( 1-t\right) \varphi \left(
y\right) \right) -h\left( t\right) g(\varphi \left( x\right)
)-h(1-t)g(\varphi \left( y\right) ) \\
&& \\
&\leq &h\left( t\right) f(\varphi \left( x\right) )+h(1-t)f(\varphi \left(
y\right) )-f\left( t\varphi \left( x\right) +\left( 1-t\right) \varphi
\left( y\right) \right) \\
&& \\
&\leq &h\left( t\right) g(\varphi \left( x\right) )+h(1-t)g(\varphi \left(
y\right) )-g\left( t\varphi \left( x\right) +\left( 1-t\right) \varphi
\left( y\right) \right)
\end{eqnarray*}%
for all $x,y\in I$ and $t\in \left[ 0,1\right] .$ The two inequalities may
be rearranged as%
\begin{eqnarray*}
&&\left( g+f\right) \left( t\varphi \left( x\right) +\left( 1-t\right)
\varphi \left( y\right) \right) \\
&& \\
&\leq &h\left( t\right) \left( g+f\right) (\varphi \left( x\right)
)+h(1-t)\left( g+f\right) (\varphi \left( y\right) )
\end{eqnarray*}%
and%
\begin{eqnarray*}
&&\left( g-f\right) \left( t\varphi \left( x\right) +\left( 1-t\right)
\varphi \left( y\right) \right) \\
&& \\
&\leq &h\left( t\right) \left( g-f\right) (\varphi \left( x\right)
)+h(1-t)\left( g-f\right) (\varphi \left( y\right) )
\end{eqnarray*}%
which are eqivalent to the $\varphi _{h}-$convexity of $g+f$ and $g-f,$
respectively.

2$\Longleftrightarrow $3 Let we define the mappings $f,$ $g$ as $f=\frac{1}{2%
}\left( l-k\right) $ and $g=\frac{1}{2}\left( l+k\right) $. Then if we sum
and subtract $f$ and $g,$ respectively, we have $g+f=l$ and $g-f=k.$ By the
condition 2 in Lemma \ref{2}, the mappings $g-f$ and $g+f$ are $\varphi
_{h}- $convex on $I,$ so $l,k$ are $\varphi _{h}-$convex mappings on $I$ too.
\end{proof}

\begin{theorem}
\label{3} Let $h:\left( 0,1\right) \rightarrow \left( 0,\infty \right) $ be
a given function, $g:I\rightarrow \lbrack 0,\infty )$ be a given $\varphi
_{h}-$convex function. If $f:I\rightarrow \lbrack 0,\infty )$ is Lebesgue
integrable and $\left( g,\varphi _{h}\right) -$convex dominated on $I$ for
linear continuous function $\varphi :\left[ a,b\right] \rightarrow \left[ a,b%
\right] ,$ then the following inequalities hold:%
\begin{eqnarray}
&&\left\vert \frac{1}{\varphi \left( b\right) -\varphi \left( a\right) }%
\int_{\varphi \left( a\right) }^{\varphi \left( b\right) }f\left( x\right)
dx-\frac{1}{2h\left( \frac{1}{2}\right) }f\left( \frac{\varphi \left(
a\right) +\varphi \left( b\right) }{2}\right) \right\vert  \label{c} \\
&&  \notag \\
&\leq &\frac{1}{\varphi \left( b\right) -\varphi \left( a\right) }%
\int_{\varphi \left( a\right) }^{\varphi \left( b\right) }g\left( x\right)
dx-\frac{1}{2h\left( \frac{1}{2}\right) }g\left( \frac{\varphi \left(
a\right) +\varphi \left( b\right) }{2}\right)  \notag
\end{eqnarray}%
and%
\begin{eqnarray}
&&\left\vert \left[ f\left( \varphi \left( a\right) \right) +f\left( \varphi
\left( b\right) \right) \right] \int_{0}^{1}h\left( t\right) dt-\frac{1}{%
\varphi \left( b\right) -\varphi \left( a\right) }\int_{\varphi \left(
a\right) }^{\varphi \left( b\right) }f\left( x\right) dx\right\vert
\label{d} \\
&&  \notag \\
&\leq &\left[ g\left( \varphi \left( a\right) \right) +g\left( \varphi
\left( b\right) \right) \right] \int_{0}^{1}h\left( t\right) dt-\frac{1}{%
\varphi \left( b\right) -\varphi \left( a\right) }\int_{\varphi \left(
a\right) }^{\varphi \left( b\right) }g\left( x\right) dx  \notag
\end{eqnarray}%
for all $x,y\in I$ and $t\in \left[ 0,1\right] .$
\end{theorem}

\begin{proof}
By the Definition \ref{1} with $t=\frac{1}{2},\ x=\lambda a+(1-\lambda )b,\
y=\left( 1-\lambda \right) a+\lambda b\ $and$\ \lambda \in \left[ 0,1\right]
,$ as the mapping $f$ is $\left( g,\varphi _{h}\right) -$convex dominated
function, we have that%
\begin{eqnarray*}
&&\left\vert h\left( \frac{1}{2}\right) \left[ f\left( \varphi \left(
\lambda a+(1-\lambda )b\right) \right) +f\left( \varphi \left( \left(
1-\lambda \right) a+\lambda b\right) \right) \right] -f\left( \frac{\varphi
\left( \lambda a+(1-\lambda )b\right) +\varphi \left( \left( 1-\lambda
\right) a+\lambda b\right) }{2}\right) \right\vert \\
&\leq & \\
&&h\left( \frac{1}{2}\right) \left[ g\left( \varphi \left( \lambda
a+(1-\lambda )b\right) \right) +g\left( \varphi \left( \left( 1-\lambda
\right) a+\lambda b\right) \right) \right] -g\left( \frac{\varphi \left(
\lambda a+(1-\lambda )b\right) +\varphi \left( \left( 1-\lambda \right)
a+\lambda b\right) }{2}\right) .
\end{eqnarray*}%
Then using the linearity of $\varphi -$function, we have%
\begin{eqnarray*}
&&\left\vert h\left( \frac{1}{2}\right) \left[ f\left( \lambda \varphi
\left( a\right) +\left( 1-\lambda \right) \varphi \left( b\right) \right)
+f\left( \left( 1-\lambda \right) \varphi \left( a\right) +\lambda \varphi
\left( b\right) \right) \right] -f\left( \frac{\varphi \left( a\right)
+\varphi \left( b\right) }{2}\right) \right\vert \\
&\leq & \\
&&h\left( \frac{1}{2}\right) \left[ g\left( \lambda \varphi \left( a\right)
+\left( 1-\lambda \right) \varphi \left( b\right) \right) +g\left( \left(
1-\lambda \right) \varphi \left( a\right) +\lambda \varphi \left( b\right)
\right) \right] -g\left( \frac{\varphi \left( a\right) +\varphi \left(
b\right) }{2}\right) .
\end{eqnarray*}%
If we integrate the above inequality with respect to $\lambda $ over $\left[
0,1\right] ,$ the inequality in (\ref{c}) is proved.

To prove the inequality in (\ref{d}), firstly we use the Definition \ref{1}
for $x=a$ and $y=b$, we have%
\begin{eqnarray*}
&&\left\vert h\left( t\right) f\left( \varphi \left( a\right) \right)
+h\left( 1-t\right) f\left( \varphi \left( b\right) \right) -f\left(
t\varphi \left( a\right) +\left( 1-t\right) \varphi \left( b\right) \right)
\right\vert \\
&\leq & \\
&&h\left( t\right) g\left( \varphi \left( a\right) \right) +h\left(
1-t\right) g\left( \varphi \left( b\right) \right) -g\left( t\varphi \left(
a\right) +\left( 1-t\right) \varphi \left( b\right) \right) .
\end{eqnarray*}%
Then, we integrate the above inequality with respect to $t$ over $\left[ 0,1%
\right] ,$ we get%
\begin{eqnarray*}
&&\left\vert f\left( \varphi \left( a\right) \right) \int_{0}^{1}h\left(
t\right) dt+f\left( \varphi \left( b\right) \right) \int_{0}^{1}h\left(
1-t\right) dt-\int_{0}^{1}f\left( t\varphi \left( a\right) +\left(
1-t\right) \varphi \left( b\right) \right) dt\right\vert \\
&& \\
&\leq &g\left( \varphi \left( a\right) \right) \int_{0}^{1}h\left( t\right)
dt+g\left( \varphi \left( b\right) \right) \int_{0}^{1}h\left( 1-t\right)
dt-\int_{0}^{1}g\left( t\varphi \left( a\right) +\left( 1-t\right) \varphi
\left( b\right) \right) dt.
\end{eqnarray*}%
If we substitute $x=t\varphi \left( a\right) +\left( 1-t\right) \varphi
\left( b\right) $ and use the fact that $\int_{0}^{1}h\left( t\right)
dt=\int_{0}^{1}h\left( 1-t\right) dt,$ we get%
\begin{eqnarray*}
&&\left\vert \left[ f\left( \varphi \left( a\right) \right) +f\left( \varphi
\left( b\right) \right) \right] \int_{0}^{1}h\left( t\right) dt-\frac{1}{%
\varphi \left( b\right) -\varphi \left( a\right) }\int_{\varphi \left(
a\right) }^{\varphi \left( b\right) }f\left( x\right) dx\right\vert \\
&& \\
&\leq &\left[ g\left( \varphi \left( a\right) \right) +g\left( \varphi
\left( b\right) \right) \right] \int_{0}^{1}h\left( t\right) dt-\frac{1}{%
\varphi \left( b\right) -\varphi \left( a\right) }\int_{\varphi \left(
a\right) }^{\varphi \left( b\right) }g\left( x\right) dx.
\end{eqnarray*}%
So, the proof is completed.
\end{proof}

\begin{corollary}
Under the assumptions of Theorem \ref{3} with $h\left( t\right) =t,$ $t\in
\left( 0,1\right) $, we have%
\begin{eqnarray}
&&\left\vert \frac{1}{\varphi \left( b\right) -\varphi \left( a\right) }%
\int_{\varphi \left( a\right) }^{\varphi \left( b\right) }f\left( x\right)
dx-f\left( \frac{\varphi \left( a\right) +\varphi \left( b\right) }{2}%
\right) \right\vert  \label{e} \\
&\leq &  \notag \\
&&\frac{1}{\varphi \left( b\right) -\varphi \left( a\right) }\int_{\varphi
\left( a\right) }^{\varphi \left( b\right) }g\left( x\right) dx-g\left( 
\frac{\varphi \left( a\right) +\varphi \left( b\right) }{2}\right)  \notag
\end{eqnarray}%
and%
\begin{eqnarray}
&&\left\vert \frac{f\left( \varphi \left( a\right) \right) +f\left( \varphi
\left( b\right) \right) }{2}-\frac{1}{\varphi \left( b\right) -\varphi
\left( a\right) }\int_{\varphi \left( a\right) }^{\varphi \left( b\right)
}f\left( x\right) dx\right\vert  \label{f} \\
&\leq &  \notag \\
&&\frac{g\left( \varphi \left( a\right) \right) +g\left( \varphi \left(
b\right) \right) }{2}-\frac{1}{\varphi \left( b\right) -\varphi \left(
a\right) }\int_{\varphi \left( a\right) }^{\varphi \left( b\right) }g\left(
x\right) dx.  \notag
\end{eqnarray}

\begin{remark}
If function $\varphi $ is the identity in (\ref{e}) and (\ref{f}), then they
reduce to Hermite-Hadamard type inequalities for convex dominated functions
proved by Dragomir, Pearce and Pe\v{c}ari\'{c} in \cite{DPP}.
\end{remark}
\end{corollary}

\begin{corollary}
Under the assumptions of Theorem \ref{3} with $h\left( t\right) =t^{s},$ $%
t,s\in \left( 0,1\right) $, we have%
\begin{eqnarray}
&&\left\vert \frac{1}{\varphi \left( b\right) -\varphi \left( a\right) }%
\int_{\varphi \left( a\right) }^{\varphi \left( b\right) }f\left( x\right)
dx-2^{s-1}f\left( \frac{\varphi \left( a\right) +\varphi \left( b\right) }{2}%
\right) \right\vert  \label{i} \\
&\leq &  \notag \\
&&\frac{1}{\varphi \left( b\right) -\varphi \left( a\right) }\int_{\varphi
\left( a\right) }^{\varphi \left( b\right) }g\left( x\right)
dx-2^{s-1}g\left( \frac{\varphi \left( a\right) +\varphi \left( b\right) }{2}%
\right)  \notag
\end{eqnarray}%
and%
\begin{eqnarray}
&&\left\vert \frac{f\left( \varphi \left( a\right) \right) +f\left( \varphi
\left( b\right) \right) }{s+1}-\frac{1}{\varphi \left( b\right) -\varphi
\left( a\right) }\int_{\varphi \left( a\right) }^{\varphi \left( b\right)
}f\left( x\right) dx\right\vert  \label{j} \\
&\leq &  \notag \\
&&\frac{g\left( \varphi \left( a\right) \right) +g\left( \varphi \left(
b\right) \right) }{s+1}-\frac{1}{\varphi \left( b\right) -\varphi \left(
a\right) }\int_{\varphi \left( a\right) }^{\varphi \left( b\right) }g\left(
x\right) dx.  \notag
\end{eqnarray}

\begin{remark}
If function $\varphi $ is the identity in (\ref{i}) and (\ref{j}), then they
reduce to Hermite-Hadamard type inequalities for $\left( g,s\right) -$convex
dominated functions proved by Kavurmac\i , \"{O}zdemir and Sar\i kaya in 
\cite{KOS}.
\end{remark}
\end{corollary}

\begin{corollary}
Under the assumptions of Theorem \ref{3} with $h\left( t\right) =\frac{1}{t}%
, $ $t\in \left( 0,1\right) $, we have%
\begin{eqnarray}
&&\left\vert \frac{4}{\varphi \left( b\right) -\varphi \left( a\right) }%
\int_{\varphi \left( a\right) }^{\varphi \left( b\right) }f\left( x\right)
dx-f\left( \frac{\varphi \left( a\right) +\varphi \left( b\right) }{2}%
\right) \right\vert  \label{k} \\
&\leq &  \notag \\
&&\frac{4}{\varphi \left( b\right) -\varphi \left( a\right) }\int_{\varphi
\left( a\right) }^{\varphi \left( b\right) }g\left( x\right) dx-g\left( 
\frac{\varphi \left( a\right) +\varphi \left( b\right) }{2}\right) .  \notag
\end{eqnarray}

\begin{remark}
If function $\varphi $ is the identity in (\ref{k}), then it reduces to
Hermite-Hadamard type inequality for $\left( g,Q(I)\right) -$convex
dominated functions proved by \"{O}zdemir, Tun\c{c} and Kavurmac\i\ in \cite%
{OTK}.
\end{remark}
\end{corollary}

\begin{corollary}
Under the assumptions of Theorem \ref{3} with $h\left( t\right) =1,$ $t\in
\left( 0,1\right) $, we have%
\begin{eqnarray}
&&\left\vert \frac{2}{\varphi \left( b\right) -\varphi \left( a\right) }%
\int_{\varphi \left( a\right) }^{\varphi \left( b\right) }f\left( x\right)
dx-f\left( \frac{\varphi \left( a\right) +\varphi \left( b\right) }{2}%
\right) \right\vert  \label{l} \\
&\leq &  \notag \\
&&\frac{1}{\varphi \left( b\right) -\varphi \left( a\right) }\int_{\varphi
\left( a\right) }^{\varphi \left( b\right) }g\left( x\right) dx-g\left( 
\frac{\varphi \left( a\right) +\varphi \left( b\right) }{2}\right)  \notag
\end{eqnarray}%
and%
\begin{eqnarray}
&&\left\vert \left[ f\left( \varphi \left( a\right) \right) +f\left( \varphi
\left( b\right) \right) \right] -\frac{1}{\varphi \left( b\right) -\varphi
\left( a\right) }\int_{\varphi \left( a\right) }^{\varphi \left( b\right)
}f\left( x\right) dx\right\vert  \label{m} \\
&\leq &  \notag \\
&&\left[ g\left( \varphi \left( a\right) \right) +g\left( \varphi \left(
b\right) \right) \right] -\frac{1}{\varphi \left( b\right) -\varphi \left(
a\right) }\int_{\varphi \left( a\right) }^{\varphi \left( b\right) }g\left(
x\right) dx.  \notag
\end{eqnarray}

\begin{remark}
If function $\varphi $ is the identity in (\ref{l}) and (\ref{m}), then they
reduce to Hermite-Hadamard type inequalities for $\left( g,P(I)\right) -$%
convex dominated functions proved by \"{O}zdemir, Tun\c{c} and Kavurmac\i\
in \cite{OTK}.
\end{remark}
\end{corollary}

\end{document}